\documentclass[12 pt]{amsart}
\usepackage{amssymb,amsmath,amsfonts,amsthm, mathrsfs,tikz, pgfplots,multicol, graphicx,stmaryrd,placeins,hyperref,placeins}
\usepackage[margin=1in]{geometry}
\newtheorem{lemma}{Lemma}
\newtheorem{theorem}{Theorem}
\newtheorem*{theorem*}{Theorem}
\newtheorem{cor}{Corollary}
\newtheorem{prop}{Proposition}

\newcommand{\mbb}[1]{\mathbb{#1}}
\newcommand{\mc}[1]{\mathcal{#1}}
\begin{document}
\title{Examples of badly approximable vectors over number fields}
\author[Robert Hines]{Robert Hines}
\address{
Department of Mathematics, University of Colorado,
Campus Box 395, Boulder, Colorado 80309-0395}
\email{robert.hines@colorado.edu}
\begin{abstract}
We consider approximation of vectors $\mathbf{z}\in F\otimes\mbb{R}\cong\mbb{R}^r\times\mbb{C}^s$ by elements of a number field $F$ and construct examples of badly approximable vectors.  These examples come from compact subspaces of $SL_2(\mc{O}_F)\backslash SL_2(F\otimes\mbb{R})$ naturally associated to (totally indefinite, anisotropic) $F$-rational binary quadratic and Hermitian forms, a generalization of the well-known fact that quadratic irrationals are badly approximable over $\mbb{Q}$.
\end{abstract}
\maketitle
\section*{Introduction}
A number field $F$ of degree $r+2s$ embeds naturally in the product of its Archimedean completions $F\otimes_{\mbb{Q}}\mbb{R}\cong\mbb{R}^r\times\mbb{C}^s$.  Given a vector $\mathbf{z}=(z_1,\ldots,z_{r+s})\in\mbb{R}^r\times\mbb{C}^s$, one can ask how well $\mathbf{z}$ can be approximated by elements of $F$. Following \cite{EGL} and \cite{KL}, we will measure the quality of approximation by
$$
\max_i\{|q_i|\}\max_i\{|q_iz_i-p_i|\}, \ p/q\in F, \ p,q\in\mc{O}_F,
$$
where $p_i$ and $q_i$ are the images of $p$ and $q$ under $r+s$ inequivalent embeddings $\sigma_i:F\to\mbb{C}$ and $|\cdot|$ is the usual absolute value in $\mbb{R}$ or $\mbb{C}$.  The measure above is meaningful in the sense that all irrational vectors have infinitely many ``good'' approximations as demonstrated by the following Dirichlet-type theorem.

\begin{theorem}[cf. \cite{Q}, Theorem 1]
There is a constant $C$ depending only on $F$ such that for any $\mathbf{z}\not\in F$
$$
\max_i\{|q_i|\}\max_i\{|q_iz_i-p_i|\}\leq C
$$
has infinitely many solutions $p/q\in F$.
\end{theorem}
In what follows, we will give some explicit examples showing that the above theorem fails if the constant is decreased, i.e. there are \textit{badly approximable} vectors, $\mathbf{z}$ such that there exists $C'>0$ with
$$
\max_i\{|q_i|\}\max_i\{|q_iz_i-p_i|\}\geq C'
$$
for all $p/q\in F$.  Our examples come from ``obvious'' compact totally geodesic subspaces of
$$
SL_2(\mc{O}_F)\backslash(\mbb{H}^2)^r\times(\mbb{H}^3)^s,
$$
where $\mbb{H}^n$ is hyperbolic $n$-space.  Namely, these examples are associated to totally indefinite anisotropic $F$-rational binary quadratic forms over any number field (Proposition \ref{prop:quad}) and totally indefinite anisotropic $F$-rational binary Hermitian forms over CM fields (Proposition \ref{prop:herm}).  Among the examples are algebraic vectors, i.e. vectors whose entries generate a non-trivial finite extension of $F$, including non-quadratic vectors in the CM case (Corollary \ref{cor:alg}).  This is interesting in light of the following variation on Roth's theorem (which can be deduced from the Subspace Theorem for number fields).
\begin{theorem}[cf. \cite{S1}, Theorem 3]
Suppose  $\mathbf{z}\not\in F$ has algebraic coordinates. Then for all $\epsilon>0$, there exists a constant $C'>0$ depending on $\mathbf{z}$ and $\epsilon$ such that
$$
\max_i\{|q_i|\}^{1+\epsilon}\max_i\{|q_iz_i-p_i|\}\geq C'
$$
for all $p/q\in F$.
\end{theorem}
The ``linear forms'' notion of badly approximable defined above implies
$$
\max_i\{|z_i-p_i/q_i|\}\geq\frac{C'}{\max_i\{|q_i|^2\}} \text{ for all } p/q\in F,
$$
which is perhaps the first notion of badly approximable that comes to mind.  The two notions are equivalent when $F$ has only one infinite place (i.e. $F=\mbb{Q}$ or $\mbb{Q}(\sqrt{-d})$) since the absolute value is multiplicative.  (The notions are also equivalent for real quadratic and complex quartic $F$, \cite{KL} Proposition A.2.)  However, in larger number fields it seems that some choice must be made and the linear choice appears naturally in the proof of Theorem \ref{thm:dani}.

Simultaneous approximation in this sense seems natural and has been explored by various authors, e.g. \cite{EGL}, \cite{Ha}, \cite{KL}, \cite{Q}, \cite{S1},\cite{B}.  Among known facts, we note that the set of badly approximable vectors has Lebesgue measure zero, full Hausdorff dimension, and is even ``winning'' when restricted to curves and various fractals in $\mbb{R}^r\times\mbb{C}^s$ (\cite{EGL}, \cite{KL}, \cite{EK}).

Finally, we note that there is an elementary proof that our examples are badly approximable, along the lines of Liouville's theorem, which we give at the end of the paper.
\section*{Notation and outline}
To fix some notation and conventions, we identify $SL_2(\mbb{C})/SU_2(\mbb{C})$ with the upper half-space model of three-dimensional hyperbolic space, $\mbb{H}^3=\{\zeta=z+tj : z\in\mbb{C},0<t\in\mbb{R}\}$ (inside the Hamiltonians), via the action
$$
g\cdot\zeta=(a\zeta+b)(c\zeta+d)^{-1}, \ 
g=\left(
\begin{array}{cc}
a&b\\
c&d\\
\end{array}
\right)\in SL_2(\mbb{C}),
$$
and $SL_2(\mbb{R})/SO_2(\mbb{R})$ with the upper half-plane, $\mbb{H}^2=\{z=x+iy\in\mbb{C} : y>0\}$, via the action
$$
g\cdot z=\frac{az+b}{cz+d}, \ 
g=\left(
\begin{array}{cc}
a&b\\
c&d\\
\end{array}
\right)\in SL_2(\mbb{R}).
$$

Let $F$ be a number field of degree $r+2s$, where $r$ and $s$ are the number of real and complex places respectively, and let $n=r+s$.  We are interested in the locally symmetric space $\Gamma\backslash G/K$ where $\Gamma=SL_2(\mc{O}_F)$ acts by left multiplication on $G=SL_2(F\otimes\mbb{R})\cong SL_2(\mbb{R})^r\times SL_2(\mbb{C})^s$ and $K\cong SO_2(\mbb{R})^r\times SU_2(\mbb{C})^s$ is a maximal compact subgroup of $G$.  Let $\{\sigma_i\}_{i=1}^n$ be the set of real embeddings along with a choice of one complex embedding from each conjugate pair, so that $\Gamma$ acts diagonally via $\{\sigma_i\}_{i}$ in the isomorphism above.

We will consider products of lines $\prod_i L_i\subseteq(\mbb{H}^2)^r\times(\mbb{H}^3)^s$ for any $F$ (respectively products of planes $\prod_i P_i\subseteq(\mbb{H}^3)^n$ in the CM case) whose image modulo $\Gamma$ is compact, along with geodesic trajectories $\Omega_{\mathbf{z}}\cdot K\subseteq(\mbb{H}^2)^r\times(\mbb{H}^3)^s$ ``aimed'' at points $\mathbf{z}=(z_i)_i\in\mbb{R}^r\times\mbb{C}^s\subseteq(\partial\mathbb{H}^2)^r\times(\partial\mathbb{H}^3)^s$, where
$$
\Omega_{\mathbf{z}}=\left\{\left(
\left(
\begin{array}{cc}
1&z_1\\
0&1\\
\end{array}
\right)
\left(
\begin{array}{cc}
e^{-t}&0\\
0&e^{t}\\
\end{array}
\right),\ldots,
\left(
\begin{array}{cc}
1&z_n\\
0&1\\
\end{array}
\right)
\left(
\begin{array}{cc}
e^{-t}&0\\
0&e^{t}\\
\end{array}
\right)
\right) : 0\leq t\in\mbb{R}
\right\}.
$$
If each $z_i\in\partial L_i$ (respectively $z_i\in\partial P_i$) then the trajectory $\Omega_{\mathbf{z}}$ is asymptotic to $\prod_i L_i$ (respectively $\prod_i P_i$) and therefore bounded modulo $\Gamma$.  The Dani correspondence of the next section tells us such $\mathbf{z}$ are badly approximable.

\section*{Dani correspondence}
The results of this section are taken from \cite{EGL}.  Theorem \ref{thm:dani} is a variation of \cite{D}, Theorem 2.20, tailored to simultaneous approximation.
\begin{theorem}[\cite{EGL}, Proposition 3.1]\label{thm:dani}
The vector $\mathbf{z}=(z_1,\ldots,z_n)\in\mbb{R}^r\times\mbb{C}^s$ is badly approximable over $F$ if and only if the geodesic trajectory
$$
\Gamma\cdot\Omega_{\mathbf{z}}\cdot K\subseteq\Gamma\backslash(\mbb{H}^2)^r\times(\mbb{H}^3)^s
$$
is bounded, where
$$
\Omega_{\mathbf{z}}=\left\{\left(
\left(
\begin{array}{cc}
1&z_1\\
0&1\\
\end{array}
\right)
\left(
\begin{array}{cc}
e^{-t}&0\\
0&e^{t}\\
\end{array}
\right),\ldots,
\left(
\begin{array}{cc}
1&z_n\\
0&1\\
\end{array}
\right)
\left(
\begin{array}{cc}
e^{-t}&0\\
0&e^{t}\\
\end{array}
\right)
\right) : 0\leq t\in\mbb{R}
\right\}.
$$
\end{theorem}
This follows in a straight-forward fashion from the following version of Mahler's compactness criterion, stated here for $SL_2$.
\begin{theorem}[\cite{EGL}, Theorem 2.2]
A subset $\Gamma\cdot\Omega\subseteq\Gamma\backslash SL_2(F\otimes\mbb{R})$ is precompact if and only if there exists $\epsilon>0$ such that
$$
\max\{\max_i\{|z_i|\},\max_i\{|w_i|\}\}\geq\epsilon, \ (\mathbf{z},\mathbf{w})=(q,p)\omega,
$$
for all $(0,0)\neq(q,p)\in\mc{O}_F^2$ and $\omega\in\Omega$.  In other words, the two-dimensional $\mc{O}_F$-modules in $(F\otimes\mbb{R})^2$ spanned by the rows of $\omega\in\Omega$ do not contain arbitrarily short vectors.
\end{theorem}

An obvious way to obtain bounded trajectories in Theorem \ref{thm:dani} is to consider those asymptotic to compact totally geodesic subspaces of $\Gamma\backslash(\mbb{H}^2)^r\times(\mbb{H}^3)^s$, which we do in the next two sections.

\section*{Totally indefinite binary quadratic forms}
As above, let $F$ be a number field, $F\otimes\mbb{R}\cong\mbb{R}^r\times\mbb{C}^s$, $\Gamma=SL_2(\mc{O}_F)$, and let
$$
Q(x,y)=(x \ y)\left(
\begin{array}{cc}
A&B/2\\
B/2&C\\
\end{array}
\right)
\left(
\begin{array}{c}
x\\
y\\
\end{array}
\right)=Ax^2+Bxy+Cy^2, \ A,B,C\in F
$$
be an $F$-rational binary quadratic form with determinant $\Delta(Q)=AC-B^2/4$.  We say $Q$ is \textit{totally indefinite} if $\sigma(\Delta)<0$ for all real embeddings $\sigma:F\to\mbb{R}$ and that $Q$ is \textit{anisotropic} if it has no non-trivial zeros in $F^2$.  Note that $Q$ is anisotropic if and only if $-\Delta(Q)$ is not a square in $F$.  Let $Q_i$ be the form obtained by applying $\sigma_i$ to the coefficients of $Q$, denote by $Z_i(Q)$ the zero set of $Q_i$
$$
Z_i(Q)=\{[z:w]\in P^1(\mbb{C}) \text{ or } P^1(\mbb{R}) : Q_i(z,w)=0\},
$$
and let $Z(Q)=\prod_iZ_i(Q)$ (a finite set of cardinality $2^n$ for totally indefinite $Q$).

The group $\Gamma$ acts on binary quadratic forms by change of variable
$$
Q^g(x,y)=(g^tQg)(x,y)=Q(ax+by,cx+dy)
$$
and also on $P^1(\mbb{R})^r\times P^1(\mbb{C})^s$ diagonally by linear fractional transformations
$$
g\cdot([z_1:w_1],\ldots,[z_n:w_n])=([a_1z_1+b_1w_1:c_1z_1+d_1w_1],\ldots,[a_nz_n+b_nw_n:c_nz_n+d_nw_n]),
$$
where $a_i=\sigma_i(a)$ and similary for $b_i$, $c_i$, and $d_i$.  These actions are compatible in the sense that $g^{-1}\cdot Z(Q)=Z(Q^g)$.  Without further remark, we identify $\mbb{R}^r\times\mbb{C}^s$ with a subset of $P^1(\mbb{R})^r\times P^1(\mbb{C})^s$ via $(z_i)_i\mapsto([z_i:1])_i$.

The following is a generalization of the fact that quadratic irrationals are badly approximable over $\mbb{Q}$ ($r=1$, $s=0$), which is usually demonstrated via continued fractions (e.g. \cite{K}, Theorem 28).  We should note that these examples can also be deduced from Theorem 6.4 of \cite{B} (with $S$ the set of infinite places and $N=1$).
\begin{prop}\label{prop:quad}
Let $Q$ be a totally indefinite anisotropic $F$-rational binary quadratic form over a number field $F$.  Then any vector $\mathbf{z}\in Z(Q)$ is badly approximable over $F$.
\end{prop}
First we establish compactness of a subspace associated to $Q$.  There are many references with discussions of compactness for anisotropic arithmetic quotients, e.g. \cite{PR}, \cite{R}, and \cite{W}.
\begin{lemma}\label{lem:Q}
Let $Q$ be a totally indefinite $F$-rational anisotropic binary quadratic form, and let $L_i$ be the line in $\mbb{H}^2$ or $\mbb{H}^3$ with endpoints $Z(Q_i)$.  Then $\pi(\prod_iL_i)$ is compact in $\Gamma\backslash(\mbb{H}^2)^r\times(\mbb{H}^3)^s$, where $\pi:(\mbb{H}^2)^r\times(\mbb{H}^3)^s\to\Gamma\backslash(\mbb{H}^2)^r\times(\mbb{H}^3)^s$ is the quotient map.
\end{lemma}
\begin{proof}[Proof of Lemma \ref{lem:Q}]
Without loss of generality, suppose $Q$ has integral coefficients, $A,B,C\in\mc{O}_F$.  Compactness of
$$
SO(Q,\mc{O}_F)\backslash SO(Q,F\otimes\mbb{R})\subseteq\Gamma\backslash SL_2(F\otimes\mbb{R}),
$$
is a consequence of Mahler's compactness criterion as follows.  For $g\in SO(Q,F\otimes\mbb{R})$ and any $(0,0)\neq(\alpha,\beta)\in\mc{O}_F^2$, the quantity $\max_i\{|\sigma_i(Q^g(\alpha,\beta))|\}$ is bounded away from zero because
$$
0\neq Q^g(\alpha,\beta)=Q(\alpha,\beta)\in\mc{O}_F,
$$
and $\mc{O}_F$ is discrete in $F\otimes\mbb{R}$.  By Mahler's criterion and continuity of $Q$ viewed as a function $(\mbb{R}^r\times\mbb{C}^s)^2\to \mbb{R}^r\times\mbb{C}^s$, $SO(Q,\mc{O}_F)\backslash SO(Q,F\otimes\mbb{R})$ is precompact.\footnote{Due to choices of left/right actions, this technically shows that $SO(F\otimes\mbb{R})/SO(\mc{O}_F)$ has compact closure in $SL_2(F\otimes\mbb{R})/\Gamma$.  However the left and right coset spaces are homeomorphic.}  The inclusion above is a closed embedding, hence its image is compact.

To get the result in the locally symmetric space, note that
$$
\pi(\prod_i L_i)=\Gamma\cdot SO(Q,F\otimes\mbb{R})g\cdot K\subseteq\Gamma\backslash G/K,
$$
where $g\in SL_2(F\otimes\mbb{R})$ is any element such that $g K\in\prod_i L_i$, and that $\Gamma\backslash G\to\Gamma\backslash G/K$ is proper.
\end{proof}

\begin{proof}[Proof of Proposition \ref{prop:quad}]
Let $L_i$ be the line in $\mbb{H}^2$ or $\mbb{H}^3$ with ideal endpoints the zeros of $Q_i$ as in the lemma.  The stabilizer $Stab_{\Gamma}(\prod_iL_i)$ acts cocompactly on $\prod_iL_i$ by the lemma above.  For $\mathbf{z}\in Z(Q)$ the distance between the geodesic trajectory $\Omega_{\mathbf{z}}\cdot K$ and $\prod_iL_i$ is bounded as their projections are asympototic in each copy of $\mbb{H}^2$ or $\mbb{H}^3$.   In the quotient $\Gamma\backslash(\mbb{H}^2)^r\times(\mbb{H}^3)^s$, the image of $\prod_iL_i$ is compact and therefore $\Gamma\cdot\Omega_{\mathbf{z}}\cdot K$ is bounded in $\Gamma\backslash(\mbb{H}^2)^r\times(\mbb{H}^3)^s$.
\end{proof}
\section*{Totally indefinite binary Hermitian forms over CM fields}
Let $F$ be a CM field (an imaginary quadratic extension of a totally real field $E$) of degree $2n$ with ring of integers $\mc{O}_F$ and let $H$ be an $F$-rational binary Hermitian form
$$
H(z,w)=(\bar{z} \ \bar{w})\left(
\begin{array}{cc}
A&B\\
\overline{B}&C\\
\end{array}
\right)
\left(
\begin{array}{c}
z\\
w\\
\end{array}
\right)=Az\bar{z}+\overline{B}z\bar{w}+B\bar{z}w+Cw\bar{w}, \ A,C\in E, \ B\in F,
$$
where the overline is ``complex conjugation'' (the non-trivial automorphism of $F/E$).  Let $H_i$ be the form obtained by applying $\sigma_i$ to the coefficients of $H$ (noting that $\sigma_i$ commutes with complex conjugation).  We say $H$ is \textit{totally indefinite} if $\sigma_i(\Delta)<0$ for all $i$, where $\Delta=\det(H)=AC-B\overline{B}$.  We say $H$ is \textit{anisotropic} if $H(p,q)\neq0$ for $(p,q)\in F^2\setminus\{(0,0)\}$.  Note that $H$ is anisotropic if and only if $-\Delta$ is not a relative norm, $-\Delta\not\in N^F_E(F)$.  Denote by $Z_i(H)$ the zero set of $H_i$,
$$
Z_i(H)=\{([z:w]\in P^1(\mbb{C}) : H_i(z,w)=0\},
$$
a circle in $P^1(\mbb{C})$, and let $Z(H)=\prod_iZ_i(H)$.  When $H$ is totally indefinite, $Z(H)\cong(S^1)^s$ is an $s$-dimensional torus.

The group $\Gamma=SL_2(\mc{O}_F)$ acts on $H$ by change of variable
$$
H^ g(z,w)=(\bar{g}^tHg)(z,w)=H(az+bw,cz+dw), \ 
g=\left(
\begin{array}{cc}
a&b\\
c&d\\
\end{array}
\right)\in SL_2(\mc{O}_F),
$$
and also on $P^1(\mbb{C})^n$ diagonally by linear fractional transformations
$$
g\cdot([z_1:w_1],\ldots,[z_n:w_n])=([a_1z_1+b_1w_1:c_1z_1+d_1w_1],\ldots,[a_nz_n+b_nw_n:c_nz_n+d_nw_n]).
$$
where $a_i=\sigma_i(a)$ and similarly for $b_i$, $c_i$, and $d_i$.  These actions are compatible in the sense that $g^{-1}\cdot Z(H)=Z(H^g)$.  As before, we include $\mbb{C}^s\hookrightarrow P^1(\mbb{C})^s$ via $(z_i)_i\mapsto([z_i:1])_i$.

The following is a generalization of the fact that zeros of anisotropic binary Hermitian forms are badly approximable over imaginary quadratic fields ($r=0$, $s=1$).  As in the case of quadratic irrationals over $\mbb{Q}$, this can be demonstrated with continued fractions when the imaginary quadratic field is Euclidean, $F=\mbb{Q}(\sqrt{-d})$, $d=1,2,3,7,11$.  Details for the imaginary quadratic case can be found in \cite{Hi}.
\begin{prop}\label{prop:herm}
If $\mathbf{z}=(z_1,\ldots,z_n)\in\mbb{C}^n$ is a zero of the totally indefinite anisotropic $F$-rational binary Hermitian form $H$, i.e. $\mathbf{z}\in Z(H)$, then $\mathbf{z}$ is badly approximable.
\end{prop}
As before, we first establish compactness of a subspace associated to $H$ (once again, cf. \cite{PR}, \cite{R}, or \cite{W}).
\begin{lemma}\label{lem:plane}
If $H$ is a totally indefinite anisotropic $F$-rational binary Hermitian form, then $\pi\left(\prod_iP_i\right)$ is compact in $\Gamma\backslash(\mbb{H}^3)^n$ where $P_i$ is the geodesic plane in the $i$th copy of $\mbb{H}^3$ whose boundary at infinity is the zero set $Z_i(H)$ and $\pi:(\mbb{H}^3)^n\to\Gamma\backslash(\mbb{H}^3)^n$ is the quotient map.
\end{lemma}
\begin{proof}[Proof of Lemma \ref{lem:plane}]
Without loss of generality, suppose $H$ has integral coefficients, $A,C\in\mc{O}_F$, $B\in\mc{O}_E$.  Compactness of
$$
SU(H,\mc{O}_F)\backslash SU(H,F\otimes\mbb{R})\subseteq\Gamma\backslash SL_2(F\otimes\mbb{R}),
$$
follows from Mahler's compactness criterion as follows.  For $g\in SU(H,F\otimes\mbb{R})$ and any $(0,0)\neq(\alpha,\beta)\in\mc{O}_F^2$, the quantity $\max_i\{|\sigma_i(H^g(\alpha,\beta))|\}$ is bounded away from zero because
$$
0\neq H^g(\alpha,\beta)=H(\alpha,\beta)\in\mc{O}_E,
$$
and $\mc{O}_E$ is discrete in $F\otimes\mbb{R}$.  By Mahler's criterion and continuity of $H$ viewed as a function $(\mbb{C}^s)^2\to\mbb{C}^s$, $SU(H,\mc{O}_F)\backslash SU(H,F\otimes\mbb{R})$ is precompact.\footnote{Once again, due to choices of left/right actions, this technically shows that $SU(H,F\otimes\mbb{R})/SU(H,\mc{O}_F)$ has compact closure in $SL_2(F\otimes\mbb{R})/\Gamma$.  However the left and right coset spaces are homeomorphic.}  The inclusion above is a closed embedding, hence its image is compact.

To get the result in the locally symmetric space, note that $\pi(\prod_i P_i)=\Gamma\cdot SU(H,F\otimes\mbb{R})g\cdot K\subseteq\Gamma\backslash G/K$ where $g\in SL_2(F\otimes\mbb{R})$ is any element such that $g K\in\prod_i P_i$, and that $\Gamma\backslash G\to\Gamma\backslash G/K$ is proper.
\end{proof}
\begin{proof}[Proof of Proposition \ref{prop:herm}]
The distance between the geodesic trajectory $\Omega_{\mathbf{z}}\cdot SU_2(\mbb{C})^n$ and $\prod_iP_i$ is bounded as their projections are asympototic in each of the $n$ copies of $\mbb{H}^3$.  In the quotient $\Gamma\backslash(\mbb{H}^3)^n$, the image of $\prod_iP_i$ is compact and therefore $\Gamma\cdot\Omega_{\mathbf{z}}\cdot SU_2(\mbb{C})^n$ is bounded in $\Gamma\backslash(\mbb{H}^3)^n$.
\end{proof}

It should be emphasized that the badly approximable product of circles $Z(H)\subseteq\mbb{C}^n$ contains non-quadratic algebraic vectors, parameterized as follows (cf. \cite{Hi}).
\begin{cor}\label{cor:alg}
Choose real algebraic numbers $\alpha_i\in[-2,2]$, $1\leq i\leq n$, $f\in F$, and a totally positive $e\in E\setminus N^F_E(F)$.  Then the vectors
$$
\mathbf{z}=(z_1,\ldots, z_n), \ z_i=f_i+\sqrt{e_i}\cdot\frac{\alpha_i\pm\sqrt{\alpha_i^2-4}}{2}
$$
are algebraic and badly approximable, where $f_i=\sigma_i(f)$, $e_i=\sigma_i(e)$.
\end{cor}
\section*{Elementary proofs}
Finally, we note that Proposition \ref{prop:quad} and \ref{prop:herm} have elementary proofs along the lines of Liouville's theorem.  Let $J$ be a totally indefinite, anisotropic, integral binary quadratic or Hermitian form.  For $\mathbf{z}\in Z(J)$ and $p/q\in F$ with $\max_i\{|z_i-p_i/q_i|\}\leq1$, we have
$$
|J_i(p_i/q_i,1)|=|J(z_i,1)-J(p_i/q_i,1)|\leq \kappa_i|z_i-p_i/q_i|
$$
for some constant $\kappa_i>0$ depending on $z_i$ and $J_i$ by the mean value theorem.  Multiplying by $|q_i|^2$ we have
$$
|J_i(p_i,q_i)|\leq\kappa_i|q_i(q_iz_i-p_i)|.
$$
Because $J$ is anisotropic and integral, for any $0<\lambda\leq\min\{ \max_i\{|a_i|\}: 0\neq a\in\mc{O}_F\}$ we have
$$
\max_i\{|J_i(p_i,q_i)|\}\geq\lambda.
$$
Hence for some $i_0$ we have
$$
\lambda\leq\kappa_{i_0}|q_{i_0}(q_{i_0}z_{i_0}-p_{i_0})|\leq\kappa_{i_0}\max_i\{|q_i|\}\max_i\{|q_iz_i-p_i|\},
$$
and $\mathbf{z}$ is badly approximable
$$
\max_i\{|q_i|\}\max_i\{|q_iz_i-p_i|\}\geq\lambda\kappa_{i_0}^{-1}.
$$

\end{document}